\documentclass[11pt, reqno]{amsart}
\usepackage{amsmath,amsopn,amssymb,amsthm,multicol}
\usepackage[cp1251]{inputenc}
\usepackage[english]{babel}
\usepackage[breaklinks=true,colorlinks=true,linkcolor=blue,citecolor=blue,urlcolor=blue]{hyperref}
\usepackage[dvipdf]{graphicx}
\usepackage{mathrsfs}

\renewenvironment{proof}[1][Proof]{\textbf{#1.} }
{\ \rule{0.5em}{0.5em}}

\newtheorem{theorem}{Theorem}

\newtheorem{lemma}{Lemma}
\newtheorem{remark}{Remark}

\begin{document}

\title[The Ricci curvature ... \dots]
{The Ricci curvature and the normalized Ricci flow on the Stiefel manifolds
$\operatorname{SO}(n)/\operatorname{SO}(n-2)$}

\author{Nurlan Abiev}
\address{N.\,A.~Abiev \newline
Institute of Mathematics NAS of the Kyrgyz Republic, Bishkek, prospect Chui, 265a, 720071, Kyrgyz Republic}
\email{abievn@mail.ru}

\begin{abstract}
We proved that on  every Stiefel manifold
$V_2\mathbb{R}^n\cong \operatorname{SO}(n)/\operatorname{SO}(n-2)$ with $n\ge 3$
the normalized Ricci flow preserves the positivity of the Ricci curvature of  invariant Riemannian metrics with positive Ricci curvature.
Moreover, the normalized Ricci flow evolves all metrics with  mixed Ricci curvature into metrics with positive Ricci curvature in finite time.
From the point of view of the theory of dynamical systems
we  proved that for every invariant set~$\Sigma$ of the normalized Ricci flow
on~$V_2\mathbb{R}^n$ defined as $x_1^{n-2}x_2^{n-2}x_3=c$, $c>0$, there exists
a smaller invariant set $\Sigma\cap \mathscr{R}_{+}$ for every $n\ge 3$, where~$\mathscr{R}_{+}$ is the domain in $\mathbb{R}_{+}^3$ responsible for
parameters $x_1, x_2, x_3>0$ of invariant Riemannian metrics on~$V_2\mathbb{R}^n$ admitting positive Ricci curvature.

\vspace{2mm} \noindent Key words and phrases:
Stiefel manifold, Riemannian metric, normalized Ricci flow, Ricci curvature, dynamical system, invariant set, singular point.
\vspace{2mm}

\end{abstract}

\maketitle

\section{Introduction}

 One of powerful tools to study the evolution of Riemannian metrics on
a given  Riemannian manifold~$\mathscr{M}^d$  is the normalized Ricci  flow equation
\begin{equation}\label{ricciflow}
  \frac{\partial}{\partial t} \bold{g}(t) = -2 \operatorname{Ric}_{\bold{g}}+
 2 \bold{g}(t) \frac{S_{\bold{g}}}{d},
\end{equation}
introduced by Hamilton~\cite{Ham} for a family
of a Riemannian metrics~$\bold{g}(t)$ on $\mathscr{M}$,
where $\operatorname{Ric}_{\bold{g}}$ and $S_{\bold{g}}$ are
respectively the Ricci tensor and the scalar curvature of a metric~$\bold{g}$.
A lot of papers are devoted to a class of Riemannian
manifolds $\mathscr{M}$ called homogeneous, on which
the isometry group $\operatorname{Isom}(\mathscr{M})$ acts transitively.
Any Riemannian homogeneous manifold $\mathscr{M}$
can be identified  (is diffeomorphic) to some homogeneous space $G/H$
with $G=\operatorname{Isom}(\mathscr{M})$ being a Lie group according to Myers and Steenrod theorem
and $H=G_m$
the  isotropy subgroup at a given point $m\in \mathscr{M}$.
A large class consists of reductive homogeneous spaces  $G/H$,
where~$G$ is a compact and semisimple Lie group and~$H$ is a connected closed subgroup of $G$.
Let  $\mathfrak{g}$ and $\mathfrak{h}$ be the corresponding Lie algebras of $G$ and $H$,
and  $\mathfrak{g}=\mathfrak{h}\oplus\mathfrak{p}$ be a reductive
decomposition of $\mathfrak{g}$ orthogonal with respect to the Killing form~$B$ defined on~$\mathfrak{g}$
so that $\operatorname{Ad}(H)\mathfrak{p}\subset \mathfrak{p}$.
Since $B$ is negative definite $-B$  defines an $\operatorname{Ad}(G)$-invariant inner product
$\langle\cdot, \cdot\rangle=-B(\cdot, \cdot)$ on $\mathfrak{g}$.
Assume that  $\mathfrak{p}$  admits a decomposition
$\mathfrak{p}_1\oplus \cdots \oplus \mathfrak{p}_k$
into  pairwise inequivalent irreducible  $\operatorname{Ad}(H)$-modules
of dimensions $d_i=\operatorname{dim}\mathfrak{p}_i$ such that
$d_1+\cdots + d_k=d=\operatorname{dim} (G/H)$.
Then any $G$-invariant symmetric covariant 2-tensor on $G/H$ can be determined as
$\left.c_1\langle\cdot, \cdot\rangle\right|_{\mathfrak{p}_1}+\cdots +
\left.c_k \langle\cdot, \cdot\rangle\right|_{\mathfrak{p}_k}
$,
where $c_1,\dots, c_k$ are some real numbers.
In particular, so is  any $G$-invariant metric
$
{\bf g}(\cdot, \cdot)=\left.x_1\langle\cdot, \cdot\rangle\right|_{\mathfrak{p}_1}+\cdots +
\left.x_k \langle\cdot, \cdot\rangle\right|_{\mathfrak{p}_k}
$,
$x_i\in \mathbb{R}_{+}$, $i=1, \dots, k$,
on $G/H$ and its Ricci tensor
$
\operatorname{Ric}_{\bf g}(\cdot, \cdot)=\left.x_1 {\bf r}_1\langle\cdot, \cdot\rangle\right|_{\mathfrak{p}_1}+\cdots +
\left.x_k {\bf r}_k\langle\cdot, \cdot\rangle\right|_{\mathfrak{p}_k}
$,
${\bf r}_i={\bf r}_i(x_1,\dots, x_k)\in \mathbb{R}$, $i=1,\dots, k$,
called the principal Ricci curvatures
(see \cite{Arvan1, BerNik2020, Kerr} and references therein for details).
The circumstance $\mathfrak{p}=\mathfrak{p}_1\oplus \cdots \oplus \mathfrak{p}_k$
allows an opportunity to split~\eqref{ricciflow} into the system
of~$k$ nonlinear autonomous  differential equations
\begin{equation}\label{decom_ricciflow_general}
\dot{x}_i=
-2 x_i \left({\bf r}_i-\frac{\sum_{i=1}^k
d_i {\bf r}_i }{\sum_{i=1}^kd_i}\right),
\qquad i=1,\dots, k.
\end{equation}

The  system~\eqref{decom_ricciflow_general} was
studied in~\cite{Ab2, Ab7, AANS1, AANS2, AN, Stat} for the cases of generalized Wallach spaces,
Stiefel manifolds and generalized flag manifolds.
In particular, an interesting question is ``whether the normalized Ricci  flow
preserve or not the positivity of the Ricci curvature of Riemannian metrics on a given manifold?''
In~\cite{Ab7, AN} some new results was obtained
concerning  this question on the Wallach spaces and generalized Wallach spaces
with coincided parameters $a_1=a_2=a_3\in (0,1/2)$.
Some extensions of these results can be found in~\cite{Caven, Gonzales}.
Recently, in~\cite{Ab24} results of~\cite{Ab7, AN} were generalized
to the general case $a_1, a_2, a_3\in (0,1/2)$ in accordance with the classification of~\cite{Nikonorov4}:

\begin{theorem}[Theorem~5 in \cite{Ab24}]\label{thm_sum_a_i<1/2}
The normalized Ricci flow~\eqref{ricciflow} evolves some invariant Riemannian metrics with  positive Ricci curvature into metrics with mixed Ricci curvature  on every generalized Wallach space with $a_1+a_2+a_3\le 1/2$.
\end{theorem}

\begin{theorem}[Theorem~6 in \cite{Ab24}]\label{thm_sum_a_i>1/2}
The following assertions hold for generalized Wallach spaces with  $a_1+a_2+a_3> 1/2$:
\begin{enumerate}
\item
The normalized Ricci flow~\eqref{ricciflow} evolves all invariant Riemannian metrics
with positive Ricci curvature into metrics with positive Ricci curvature if $\theta \ge \max\left\{\theta_1, \theta_2, \theta_3\right\}$;
\item
The normalized Ricci flow~\eqref{ricciflow} evolves some metrics with positive Ricci curvature into metrics with positive Ricci curvature
if   $\theta \ge \max\left\{\theta_1, \theta_2, \theta_3\right\}$ fails,
\end{enumerate}
where $\theta:=a_1+a_2+a_3-1/2$ and $\theta_i:=a_i-\dfrac{1}{2}+\dfrac{1}{2}\sqrt{\dfrac{1-2a_i}{1+2a_i}}$,
$i=1,2,3$.
 \end{theorem}

\begin{theorem}[Theorem~7 in \cite{Ab24}]\label{thm_3}
The following assertions hold for\, $\operatorname{SO}(k+l+m)/\operatorname{SO}(k)\times \operatorname{SO}(l)\times \operatorname{SO}(m)$,
 $k\ge l\ge  m>1$ (shortly $\operatorname{GWS}$~{\bf 1} according to~\cite{Nikonorov4}).
Under the normalized Ricci flow~\eqref{ricciflow}
\begin{enumerate}
\item
All invariant Riemannian metrics with
positive Ricci curvature  can be evolved  into metrics  with
positive Ricci curvature    if either  $k\le 11$ or one of
the conditions  $2<l+m\le X(k)$ or  $l+m\ge Y(k)$ is satisfied at each  fixed
$k\in \{12, 13,14,15,16\}$;
\item
At least some invariant Riemannian metrics  with
positive Ricci curvature   can be evolved  into metrics  with
positive Ricci curvature   if $k\ge 17$ or if
$X(k)<l+m<Y(k)$  for $k\in \{12, 13,14,15,16\}$;
\item
The  number  of $\operatorname{GWS}$~{\bf 1} on which any original metric with $\operatorname{Ric}>0$ maintains $\operatorname{Ric}>0$  is finite,
whereas there are infinitely (countably)
many $\operatorname{GWS}$~{\bf 1} on which  $\operatorname{Ric}>0$ can be preserved at least for some original metrics with $\operatorname{Ric}>0$,
\end{enumerate}
where
$X(k):=\dfrac{2k(k-2)}{k+2+\sqrt{k^2-12k+4}}-k+2$ and
$Y(k):=\dfrac{2k(k-2)}{k+2-\sqrt{k^2-12k+4}}-k+2$.
\end{theorem}

Since $(l,m)\ne (1,1)$
the case of the space $\operatorname{SO}(n)/\operatorname{SO}(n-2)\times \operatorname{SO}(1)\times \operatorname{SO}(1)$ can not be covered by Theorem~\ref{thm_3}.
In the present paper we study the above question on the maintenance of positivity
of the Ricci curvature relatively  to the
homogeneous space~$\operatorname{SO}(n)/\operatorname{SO}(n-2)$
diffeomorphic to the Stiefel manifold~$V_2\mathbb{R}^n$.
In general  $V_k\mathbb{R}^n$, $k\le n$, can be defined as the set of
$n\times k$ matrices $A$ with real entries such that $A^tA=I_k$,
where $A^t$ means the conjugate transpose of $A$ and~$I_k$ is the $k\times k$ identity matrix.
$V_k\mathbb{R}^n$ becomes a compact smooth manifold in the subspace topology inherited from the topology of~$\mathbb{R}^{n\times k}$.
It is known that for $k<n$ the group~$\operatorname{SO}(n)$ acts on $V_k\mathbb{R}^n$ transitively
and $V_k\mathbb{R}^n$ is diffeomorphic to the reductive homogeneous space~$\operatorname{SO}(n)/\operatorname{SO}(n-k)$.
In the case of these spaces there are equivalent submodules
which may cause a complicated decomposition of the Ricci tensor.
Although the decomposition of~$\mathfrak{p}$ admits two equivalent submodules
$\mathfrak{p}_1$ and~$\mathfrak{p}_2$ causing the dependence of
an $\operatorname{SO}(n)$-invariant metric
of $\operatorname{SO}(n)/\operatorname{SO}(n-2)$
on four parameters it was proved in~\cite{Kerr} that  $\operatorname{SO}(n)$-invariant metrics on $\operatorname{SO}(n)/\operatorname{SO}(n-2)$ can be reduced to a diagonal form and described by an $\operatorname{Ad}(\operatorname{SO}(n-2))$-invariant inner product
\begin{equation}\label{metricStiffel_V_2(R^n)}
(\cdot, \cdot)=\left.x_1\langle\cdot, \cdot\rangle\right|_{\mathfrak{p}_1}+
\left.x_2\langle\cdot, \cdot\rangle\right|_{\mathfrak{p}_2}+
\left.x_3\langle\cdot, \cdot\rangle\right|_{\mathfrak{p}_3}, \quad x_i\in \mathbb{R}_{+},
\end{equation}
on $\mathfrak{p}=\mathfrak{p}_1\oplus \mathfrak{p}_2\oplus \mathfrak{p}_3$
with the corresponding Ricci tensor
$
\operatorname{Ric}_{\bf g}(\cdot, \cdot)=
\left.x_1 {\bf r}_1\langle\cdot, \cdot\rangle\right|_{\mathfrak{p}_1}+
\left.x_2 {\bf r}_2\langle\cdot, \cdot\rangle\right|_{\mathfrak{p}_2}+
\left.x_3 {\bf r}_3\langle\cdot, \cdot\rangle\right|_{\mathfrak{p}_3}
$
(an exposition in detail can be found in~\cite{Stat16}).
Clearly  $d_1=\operatorname{dim}\mathfrak{p}_1=n-2=\operatorname{dim}\mathfrak{p}_2=d_2$ and $d_3=\operatorname{dim}\mathfrak{p}_3=1$ with $d=2n-3$.

\medskip

The  main result of this paper is contained in the following theorem.

\begin{theorem}\label{thm_371124}
On  every Stiefel manifold
$\operatorname{SO}(n)/\operatorname{SO}(n-2)$ with $n\ge 3$
the normalized Ricci flow~\eqref{ricciflow}
evolves any  metric~\eqref{metricStiffel_V_2(R^n)} into metrics with positive Ricci curvature. \end{theorem}


\section{Results}

The following expressions  were found in~\cite{Stat}
for the principal Ricci curvatures of the metric~\eqref{metricStiffel_V_2(R^n)}:
\begin{eqnarray}\label{Riccival}\notag
{\bf r}_1 &=& \dfrac{1}{2x_1}-\dfrac{1}{4(n-2)}
\left(\dfrac{x_3}{x_1x_2}+\dfrac{x_2}{x_1x_3} -\dfrac{x_1}{x_2x_3}\right), \notag \\
{\bf r}_2 &=& \dfrac{1}{2x_2}-\dfrac{1}{4(n-2)}
\left(\dfrac{x_3}{x_1x_2}+\dfrac{x_1}{x_2x_3} -\dfrac{x_2}{x_1x_3}\right),  \\
{\bf r}_3 &=&\frac{1}{2x_3}-\dfrac{1}{4}
\left(\dfrac{x_1}{x_2x_3}+\dfrac{x_2}{x_1x_3} -\dfrac{x_3}{x_1x_2}\right). \notag
\end{eqnarray}

Then the scalar curvature $S_{\bf g}=d_1{\bf r}_1+d_2 {\bf r}_2+d_3{\bf r}_3$ takes the form
\begin{equation}\label{scalcurv}
S_{\bf g}=\dfrac{1}{2x_3}-\dfrac{1}{4}
\left(\dfrac{x_1}{x_2x_3}+\dfrac{x_2}{x_1x_3} +\dfrac{x_3}{x_1x_2}\right)
+\dfrac{n-2}{2}\left(\dfrac{1}{x_1}+\dfrac{1}{x_2}\right).
\end{equation}

Substituting  \eqref{Riccival} and \eqref{scalcurv} into~\eqref{decom_ricciflow_general}
the  following system of  ordinary differential equations can be obtained
on~$\operatorname{SO}(n)/\operatorname{SO}(n-2)$, $n\ge 3$:
\begin{equation}\label{three_equat_Stif}
\dot{x}_1 =f_1(x_1,x_2,x_3), \qquad
\dot{x}_2 =f_2(x_1,x_2,x_3), \qquad
\dot{x}_3 =f_3(x_1,x_2,x_3),
\end{equation}
where
\begin{eqnarray*}\label{f_i}\notag
f_1 &=& \frac{(n-1)\big(x_3^2+x_2^2\big)-(3n-5)x_1^2-2(n-2)
\Big[x_1x_2+(n-2)x_3x_1-(n-1)x_3x_2\Big]}{2(n-2)(2n-3)x_2x_3},  \\\notag
f_2 &=& \frac{(n-1)\big(x_3^2+x_1^2\big)-(3n-5)x_2^2-2(n-2)
\Big[x_1x_2+(n-2)x_3x_2-(n-1)x_3x_1\Big]}{2(n-2)(2n-3)x_1x_3}, \\
f_3 &=&\frac{(n-2)(x_1-x_2)^2+(n-2)(x_1+x_2)x_3-(n-1)x_3^2}{(2n-3)x_1x_2}.\notag
\end{eqnarray*}

\begin{figure}[h!]
\centering
\includegraphics[width=0.45\textwidth]{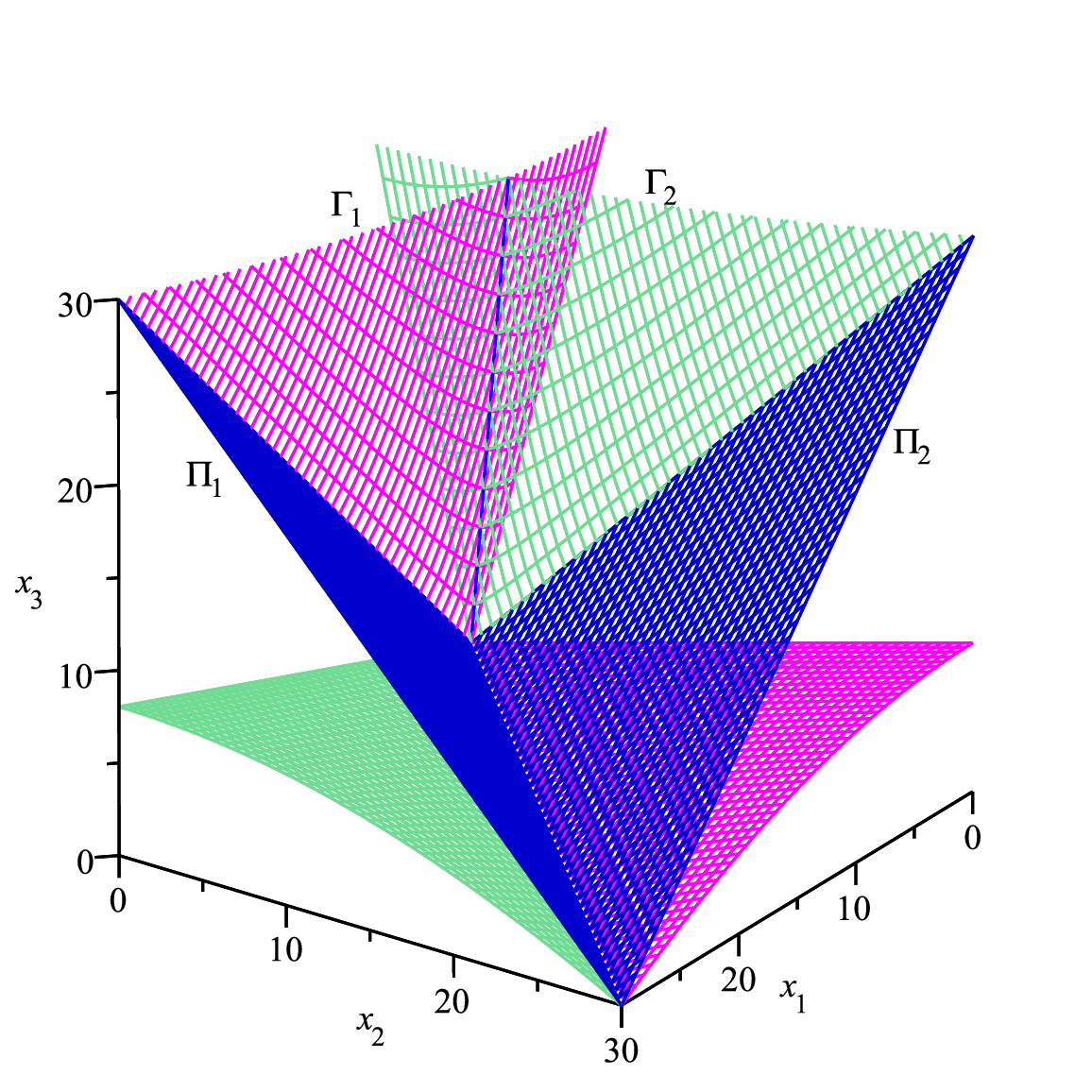}
\includegraphics[width=0.45\textwidth]{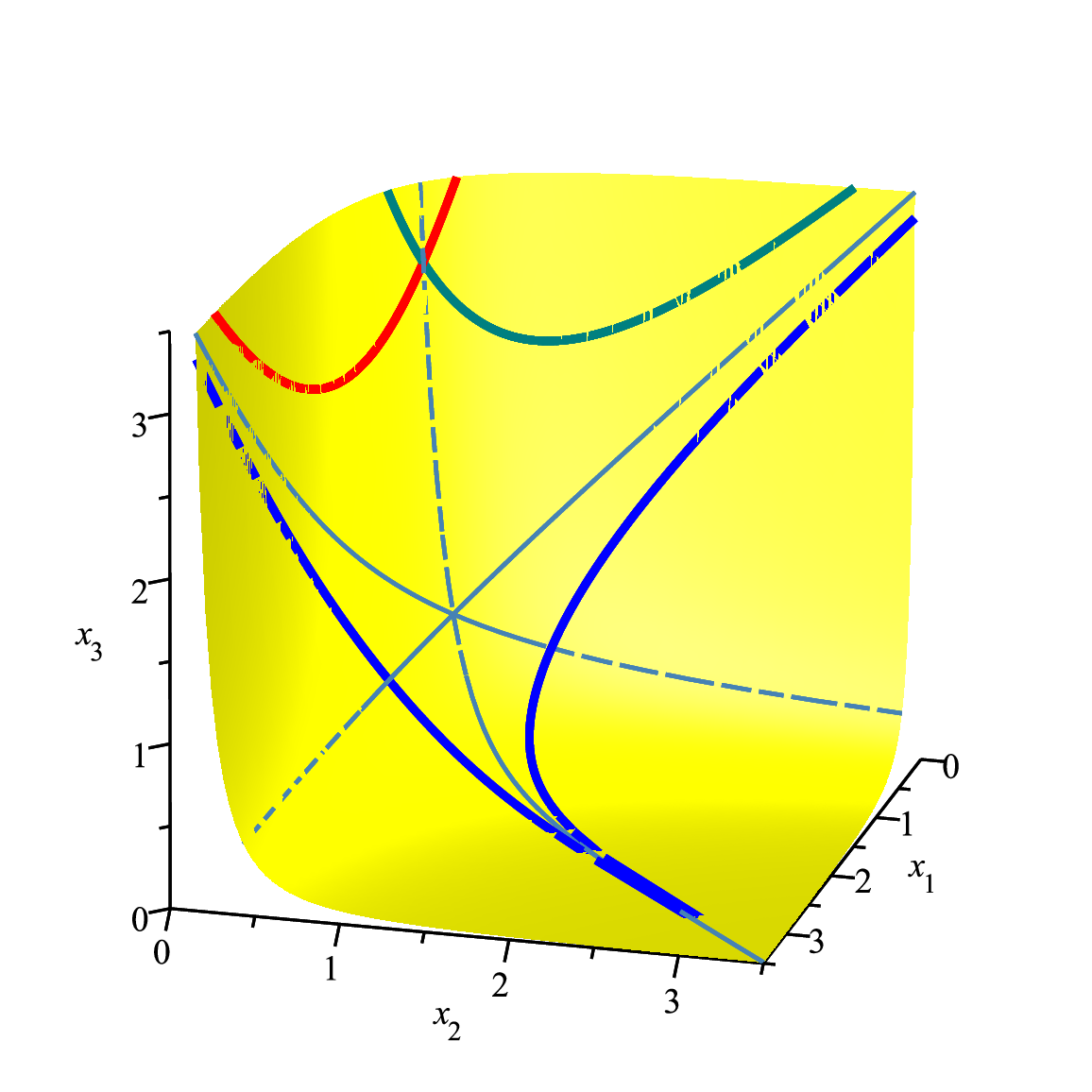}
\caption{The case $n=4$: the surfaces $\Gamma_1$, $\Gamma_2$ and
$\Pi_1, \Pi_2$ (the left panel); the  curves $\Sigma\cap \Gamma_1$, $\Sigma \cap\Gamma_2$ and
$\Sigma\cap \Pi_1$, $\Sigma \cap \Pi_2$ (the right panel)}
\label{pic-1}
\end{figure}

\subsection{The set  $\mathscr{R}_{+}$ and its structural properties}
Denote by $\mathscr{R}_{+}$ the set of metrics~\eqref{metricStiffel_V_2(R^n)}
which admit positive Ricci curvature
$$
\mathscr{R}_{+}=\big\{(x_1,x_2,x_3)\in \mathbb{R}_{+}^3 ~~ \big| ~~ {\bf r}_1>0, ~ {\bf r}_2>0, ~ {\bf r}_3>0\big\}.
$$
Observe that ${\bf r}_3 =0$ is equivalent to $(x_1-x_2+x_3)(x_2-x_1+x_3)=0$ for every $n\ge 3$.
Therefore~$\mathscr{R}_{+}$ is  bounded by two planes
$-x_1+x_2+x_3=0$ and $x_1-x_2+x_3=0$ in any case, where $x_1,x_2,x_3>0$. Denote them
respectively by~$\Pi_1$ and $\Pi_2$.
Let $\Gamma_1$ and $\Gamma_2$ be those components of the cones ${\bf r}_1=0$ and ${\bf r}_2=0$
 defined by the equations (see Figure~\ref{pic-1})
\begin{equation}\label{eqn_181124}
x_3=(n-2)x_2+\sqrt{x_1^2+(n-1)(n-3)x_2^2} \mbox{~~ and ~~}
x_3=(n-2)x_1+\sqrt{x_2^2+(n-1)(n-3)x_1^2}
\end{equation}
respectively.
Their another components
$$
x_3=(n-2)x_2-\sqrt{x_1^2+(n-1)(n-3)x_2^2} \mbox{~~ and ~~}
x_3=(n-2)x_1-\sqrt{x_2^2+(n-1)(n-3)x_1^2}.
$$
let us denote by $\Gamma_1^{-}$ and $\Gamma_2^{-}$
(the lower surfaces  in Figure~\ref{pic-1} respectively in magenta and aquamarine).
Let~$\Pi_3$ be the plane $x_1+x_2-x_3=0$ and $L$ be the straight line
$\big(p, p, (2n-4)p\big)$, $p>0$  (depicted in blue color in  Figure~\ref{pic-1}).

\begin{lemma}\label{lem_Stif_surfes}
For all $n\ge 3$
the set  $\mathscr{R}_{+}$ has the boundary
$
\partial (\mathscr{R}_{+})=
\begin{cases}
\Pi_1 \cup\Pi_2 \cup \widetilde{\Gamma}_1 \cup\widetilde{\Gamma}_2,      ~&\mbox{if} \;\; n\ge 4, \\
\Pi_1 \cup\Pi_2\cup \Pi_3, ~&\mbox{if} \;\; n=3
\end{cases}
$
with
$\widetilde{\Gamma}_1=\left\{(\mu t, \nu t, t) ~ \big| ~ t>0,\, 0<\nu\le  \widetilde{\nu}\right\}\subset \Gamma_1$ and
$\widetilde{\Gamma}_2=\left\{(\nu t, \mu t, t) ~ \big| ~ t>0,\, 0<\nu \le \widetilde{\nu}\right\}\subset \Gamma_2$,
where
$\mu=\sqrt{\nu^2-2(n-2)\nu+1}$ and $\widetilde{\nu}=(2n-4)^{-1}$.
Moreover,
$\Pi_1 \cap \Pi_2= \Pi_1 \cap \widetilde{\Gamma}_1= \Pi_2 \cap \widetilde{\Gamma}_2=\emptyset$ and
$\widetilde{\Gamma}_1\cap \widetilde{\Gamma}_2=L$ for all $n\ge 3$.
\end{lemma}

\begin{proof}
Let us consider $\Gamma_1\cap \Gamma_2$.
The system of equations
$
\begin{cases}
x_1^2-x_2^2-x_3^2+2(n-2)x_2x_3=0, \\
x_1^2-x_2^2+x_3^2-2(n-2)x_1x_3=0
\end{cases}
$
is equivalent to the system of ${\bf r}_1=0$ and ${\bf r}_2=0$.
It is easy to see that  $x_3=(n-2)(x_1+x_2)$ is a consequent of this system
corresponding to  $\Gamma_1\cap \Gamma_2$.
Substituting it into the second equation in the system
we obtain
$$
(n-1)(n-3)(x_1+x_2)(x_1-x_2)=0.
$$
In what follows that $x_1=x_2$ if $n\ge  4$, hence the straight line
$x_1=x_2=p>0$, $x_3=2(n-2)p$ denoted by~$L$ is a common part of the cones $\Gamma_1$ and $\Gamma_2$.
If $n=3$ then the plane $\Pi_3$ will replace both $\Gamma_1$ and~$\Gamma_2$
(the identity $\Gamma_1=\Gamma_2$  follows also from~\eqref{eqn_181124} at $n=3$).
The  system above admits also the solution $x_1= x_2$, $x_3=0$ non permissible for our aims.
Obviously it corresponds to  $\Gamma_1^{-}\cap \Gamma_2^{-}$.

The component $\Gamma_1$ of the cone ${\bf r}_1=0$ can be parameterized as
\begin{equation*}\label{param_cone_stif}
x_1=\mu t, \quad x_2=\nu t, \quad x_3=t,   \qquad t>0, \qquad 0<\nu < l,
\end{equation*}
where $\mu=\sqrt{\nu^2-2(n-2)\nu+1}$ and
\begin{equation}\label{eqn_211124}
l=n-2-\sqrt{(n-1)(n-3)}
\end{equation}
is the smallest root of the quadratic equation $\nu^2-2(n-2)\nu+1=0$ (by symmetry for the component~$\Gamma_2$ of~${\bf r}_2=0$  we have $x_2=\mu t$, $x_1=\nu t$, $x_3=t$).

For all $n\ge 4$ the conic components $\Gamma_1$ and $\Gamma_2$ have  extra pieces  not relating to the set  $\mathscr{R}_{+}$.
Since $\Gamma_1 \cap \Gamma_2=L$
then necessarily $\mu=\nu$ implying the critical value
$\widetilde{\nu}=(2n-4)^{-1}$.
It is easy to see that the useful part $\widetilde{\Gamma}_1$ of $\Gamma_1$ can be parameterized
by values $t>0$ and $0<\nu\le \widetilde{\nu}$,
as for the extra part of~$\Gamma_1$  then $\widetilde{\nu} <\nu < l$ and $t>0$
correspond to it.
The description of $\widetilde{\Gamma}_2$ is obvious from symmetry.
The equality $\widetilde{\Gamma}_1 \cap \widetilde{\Gamma}_2=L$ is obvious too.
The system of the equations $-x_1+x_2+x_3=0$ and $x_1-x_2+x_3=0$ implies~$x_3=0$.
Analogously, it follows from $\Pi_1 \cap \Gamma_1$ and $\Pi_2 \cap \Gamma_2$ that $x_2=0$ and $x_1=0$. Therefore $\Pi_1 \cap \Pi_2= \Pi_1 \cap \Gamma_1= \Pi_2 \cap \Gamma_2=\emptyset$
for $x_i>0$ implying $\Pi_1 \cap \Pi_2= \Pi_1 \cap \widetilde{\Gamma}_1= \Pi_2 \cap \widetilde{\Gamma}_2=\emptyset$.
\end{proof}

\subsection{The vector field ${\bf V}$ on the boundary  $\partial(\mathscr{R}_{+})$}
Introduce the vector field
$$
{\bf V} \colon {\bf x}\mapsto \big(f_1({\bf x}), f_2({\bf x}), f_3({\bf x})\big),
\quad {\bf x}=(x_1,x_2,x_3)\in \mathbb{R}_{+}^3,
$$
associated with the differential system~\eqref{three_equat_Stif}.

\begin{lemma}\label{lem_241024_2}
For every $n\ge 3$ the vector field ${\bf V}$
associated with the differential system~\eqref{three_equat_Stif}
is directed into the domain $\mathscr{R}_{+}$ on every point of its
boundary~$\partial(\mathscr{R}_{+})$.
\end{lemma}

\begin{proof}
By Lemma~\ref{lem_Stif_surfes} the planes $\Pi_1$ and $\Pi_2$
bound the domain $\mathscr{R}_{+}$ for every $n\ge 3$.
Denote by ${\bf n}_1$ and~${\bf n}_2$  their normals
$(-1,1,1)$ and $(1,-1,1)$.
We claim that the inner product
$({\bf V}, {\bf n}_1)$ is positive on every point of~$\Pi_1$ for all $n\ge 3$.
Indeed $\Pi_1$ can be parameterized as $x_2=u$, $x_3=v$, $x_1=u+v$ with $u,v>0$.
Then
$$
({\bf V}, {\bf n}_1)=-f_1\big|_{\Pi_1}+f_2\big|_{\Pi_1}
+f_3\big|_{\Pi_1}=\frac{2}{n-2}>0
$$
on $\Pi_1$ for $n\ge 3$, where
\begin{eqnarray*}\label{f_p}\notag
f_1\big|_{\Pi_1} &=& \frac{(n-2)(n-3)v-(3n-5)u}{(2n-3)(n-2)u},\notag \\
f_2\big|_{\Pi_1} &=&\frac{(n-1)\big(u-(n-3)v\big)}{(2n-3)(n-2)(u+v)},  \\
f_3\big|_{\Pi_1} &=& \frac{v}{2n-3}\frac{2(n-2)u+(n-3)v}{(u+v)u}. \notag
\end{eqnarray*}

The same inequality
$({\bf V}, {\bf n}_2)=\frac{2}{n-2} >0$ holds on the plane $\Pi_2$ for all $n\ge 3$.
Indeed $\Pi_2$ can be parameterized as  $x_1=u$, $x_3=v$, $x_2=u+v$ and
$f_1\big|_{\Pi_2}= f_2\big|_{\Pi_1}$,
$f_2\big|_{\Pi_2}= f_1\big|_{\Pi_1}$,
$f_3\big|_{\Pi_2}= f_3\big|_{\Pi_1}$.

\smallskip

{\it The case $n=3$}.  According to Lemma~\ref{lem_Stif_surfes} we have one more plane $\Pi_3$
that bounds the domain $\mathscr{R}_{+}$. Note that $f_1, f_2$ and $f_3$ take the following forms on $\Pi_3$:
$$
f_1\big|_{\Pi_3} =\frac{u}{u+v}, \qquad f_2\big|_{\Pi_3}=\frac{v}{u+v}, \qquad
f_3\big|_{\Pi_3} = -\frac{4}{3}.
$$
Then we obtain
$({\bf V}, {\bf n}_3)=f_1\big|_{\Pi_3}+f_2\big|_{\Pi_3}
-f_3\big|_{\Pi_3}=2 >0$ on $\Pi_3$, where ${\bf n}_3=(1,1,-1)$ is the normal
of $\Pi_3$.

Since the normals ${\bf n}_1$,   ${\bf n}_2$  and ${\bf n}_3$
are directed inside the domain~$\mathscr{R}_{+}$
the inequalities $({\bf V}, {\bf n}_i) >0$ imply
that the vector field ${\bf V}$ is directed towards~$\mathscr{R}_{+}$
on every point of the boundary $\partial(\mathscr{R}_{+})= \Pi_1\cup\Pi_2\cup \Pi_3$.

\smallskip

{\it The case $n\ge 4$}. Then $\partial (\mathscr{R}_{+})= \Pi_1 \cup\Pi_2 \cup \widetilde{\Gamma}_1 \cup \widetilde{\Gamma}_2$ by Lemma~\ref{lem_Stif_surfes}.
It suffices consider the surface
$\widetilde{\Gamma}_1=\left\{(\mu t, \nu t, t)~ | ~ t>0,\, \nu\in \big(0, \widetilde{\nu}\big]\right\}$,
where $\widetilde{\nu}=(2n-4)^{-1}$.

The components of the normal ${\bf m}=\nabla {\bf r}_1=\left(\frac{\partial{\bf r}_1}{\partial x_1}, \frac{\partial{\bf r}_1}{\partial x_2}, \frac{\partial{\bf r}_1}{\partial x_3}\right)$ to the conic component~$\Gamma_1$ have the forms
\begin{eqnarray*}
\frac{\partial{\bf r}_1}{\partial x_1}&=&2x_1>0, \\
\frac{\partial{\bf r}_1}{\partial x_2}&=&2(n-2)x_3-2x_2, \\
\frac{\partial{\bf r}_1}{\partial x_2}&=&2(n-2)x_2-2x_3.
\end{eqnarray*}
It is easy to check that
$ \frac{1}{2(n-2)}< n-2-\sqrt{(n-1)(n-3)}\le \frac{1}{n-2}$ for $n\ge 3$.
Then the inequalities $0<\nu\le \widetilde{\nu} <l< \frac{1}{n-2}<n-2$
yield
$$
\frac{\partial{\bf r}_1}{\partial x_2}=2t\,(n-2-\nu)>0, \qquad
\frac{\partial{\bf r}_1}{\partial x_3}=2t\,\big((n-2)\nu-1\big)<0
$$
on $\Gamma_1$ and, in particular, on $\widetilde{\Gamma}_1$
 implying that  ${\bf m}$ is directed inside the domain $\mathscr{R}_{+}$ for all $n\ge 4$.

We claim that the inner product $({\bf V}, {\bf m})$ is
positive on every point of~$\widetilde{\Gamma}_1$.
Indeed by calculations in Maple we obtain that
$\operatorname{sign}({\bf V}, {\bf m})=\operatorname{sign} (F-G)$,
where
\begin{eqnarray*}
F=F(\nu)&:=&-(n-1)(n-2-\nu)\big((n-2)\nu-1\big), \\
G=G(\nu)&:=&(n-2)(n-3)(\nu+1)\sqrt{\nu^2-2(n-2)\nu+1}.
\end{eqnarray*}
Since $F>0$ and $G>0$ for all $n\ge 4$ and all  $\nu\in \big(0, \widetilde{\nu}\big]$
the sign of $F-G$ coincides with the sign of the polynomial
\begin{multline*}
p(\nu):=F^2-G^2=4(n-2)^3\nu^4-2(n-2)(5n^3-31n^2+67n-49)\nu^3\\+
(n^6-6n^5-5n^4+136n^3-453n^2+630n-327)\nu^2\\
-2(n-2)(5n^3-31n^2+67n-49)\nu+4(n-2)^3.
\end{multline*}

We claim that $p(\nu)>0$ for all $\nu\in \big(0, \widetilde{\nu}\big]$ and all $n\ge 4$.
It is more convenient to consider the interval  $[0,1]\supset  \big(0, \widetilde{\nu}\big]$
instead of $\big(0, \widetilde{\nu}\big]$.
Since the discriminant
$$
(n-2)^4(n-3)^9(n-1)^{10}(4n^2-13n+11)^2\left[8+20(n-3)+12(n-3)^2+(n-3)^3\right]
$$
of the polynomial $p(\nu)$
is positive for every $n\ge 4$ then $p(\nu)$ has no multiple roots in $[0,1]$.
Therefore  we can use Sturm's method for our goal.
Construct the sequence of  Sturm polynomials
$p_0=p$, $p_1=p_0'$, $p_{i+1}=-\operatorname{rem}(p_{i-1},p_i)$,
where $\operatorname{rem}(p_{i-1},p_i)$ is the remainder of the division of $p_{i-1}$
by $p_i$, $i=1,2,3$. Then for $\nu=0$ we obtain using Maple that
\begin{eqnarray*}
p_0(0) &=& 4(n-2)^3>0,\\
p_1(0) &=& -2(n-2)\alpha_1<0,\\
p_2(0) &=& \frac{(n-1)(5n-9)(n-3)^2(5n^2-18n+17)}{16(n-2)}>0,\\
p_3(0) &=& 32(n-1)(n-2)^2(n-3)(4n^2-13n+11)\, \frac{\alpha_1\alpha_2}{\alpha_3^2}>0,\\
p_4(t)&\equiv & \frac{(n-1)^6(n-2)(n-3)^3}{16}
\left(\frac{\alpha_3}{\alpha_4}\right)^2\alpha_5>0,
\end{eqnarray*}
where
\begin{eqnarray*}
\alpha_1 &=& 5(n-3)^3+14(n-3)^2+16(n-3)+8>0,\\
\alpha_2 &=& (n-3)^3+10(n-3)^2+20(n-3)+12>0,\\
\alpha_3&=&8n^4-83n^3+285n^2-413n+219\ne 0~ (\mbox{the divisor 73 of 219 is not a root}),  \\
\alpha_4&=&n^9-5n^8-24n^7+256n^6-802n^5+930n^4+692n^3-3268n^2+3589n-1401\ne 0,\\
\alpha_5&=&(n-3)^3+12(n-3)^2+20(n-3)+8>0
\end{eqnarray*}
for $n\ge 4$, where
$\alpha_4\ne 0$ for $n\ge 4$ because of the unique acceptable divisor $n=467$ of $1401$
does not satisfy $\alpha_4=0$.

Evaluate now $p_i$ at $\nu=1$:
\begin{eqnarray*}
p_0(1) &=& (n-3)^3 \alpha_5>0,\\
p_1(1) &=& 2(n-3)^3 \alpha_5>0,\\
p_2(1) &=& -\frac{(n-1)(n-3)^2(5n^2-18n+17) \alpha_5}{16(n-2)^2}<0,\\
p_3(1) &=&- 32(n-1)(n-2)(n-3)^2(4n^2-13n+11)\, \frac{\alpha_5\alpha_6}{\alpha_3^2}<0,\\
\end{eqnarray*}
where
$\alpha_6=24+60(n-3)+63(n-3)^2+35(n-3)^3+10(n-3)^4+(n-3)^5>0$.

It follows then   the Sturm sequence
has the same number of sign changes equal to $1$ for both $\nu=0$ and $\nu=1$
under the condition $n\ge 4$.
Therefore $p(\nu)$ has no real roots in the interval $[0,1]$ for all $n\ge 4$.
Since $p(0)>0$ we have  $p(\nu)>0$ for $\nu\in [0,1]$,
in particular, for  $\nu\in \big(0, \widetilde{\nu}\big]$.
Therefore $({\bf V}, {\bf m})>0$ on~$\widetilde{\Gamma}_1$.
Since the normal ${\bf m}$
is directed into the domain~$\mathscr{R}_{+}$ as shown above
the vector field~${\bf V}$ is directed towards~$\mathscr{R}_{+}$
on the surface~$\widetilde{\Gamma}_1$.
By symmetry the similar assertion holds for ${\bf V}$ on the surface~
$\widetilde{\Gamma}_2\subset \Gamma_2$.
\end{proof}

\subsection{Invariant sets and singular points of the system~\eqref{three_equat_Stif}}

Let $\Sigma$ be the surface defined by the equation $\operatorname{Vol}(x_1,x_2,x_3)=c$, $c>0$,
where
$$
\operatorname{Vol}(x_1,x_2,x_3)=x_1^{n-2}x_2^{n-2}x_3
$$
is the volume function for the metric~\eqref{metricStiffel_V_2(R^n)}.
Introduce  also the following subsets of $\Sigma$  (curves on $\Sigma$):
\begin{eqnarray*}
I_1&=&\left\{\left(c^{\tfrac{1}{n-2}}\,\tau^{\tfrac{1-n}{n-2}}, \tau, \tau\right)\in \mathbb{R}_{+}^3 ~ \big| ~ \tau>0\right\}, \\
I_2&=&\left\{\left(\tau, c^{\tfrac{1}{n-2}}\, \tau^{\tfrac{1-n}{n-2}}, \tau\right)\in \mathbb{R}_{+}^3
~ \big| ~ \tau>0\right\}, \\
I_3&=&\left\{\left(\tau, \tau, c\tau^{4-2n}\right)\in \mathbb{R}_{+}^3 ~ \big| ~ \tau>0\right\}.
\end{eqnarray*}

In Figure~\ref{pic-1} the curves $I_1, I_2, I_3$ are depicted in gold color for $c=1$ and $n=4$.

\begin{lemma}\label{invar_Stiefel}
The following assertions hold for the system~\eqref{three_equat_Stif}:
\begin{enumerate}
\item
The algebraic surface $\Sigma$ is invariant for every $n\ge 3$ and every $c>0$;
\item
The curve $I_3$ is invariant for  every $n\ge 3$ and every $c>0$;
\item
The curves $I_1$ and $I_2$ are also invariant for every $c>0$ if $n=3$.
\item
For all  $n\ge 3$ the inclusions are true:
$$
\begin{cases}
I_1, I_2\subset \mathscr{R}_{+}, &\mbox{if} \;\; \tau>\tau_1:=2^{2-n}c, \\
I_3\subset \mathscr{R}_{+}, &\mbox{if} \;\; \tau>\tau_2:=\left(\dfrac{2n-4}{c}\right)^{\tfrac{1}{3-2n}}.
\end{cases}
$$
\end{enumerate}
\end{lemma}

\begin{proof}
1) Actually the first assertion can easily be proved for  system~\eqref{decom_ricciflow_general}
in general case. Put $\operatorname{Vol}=x_1^{d_1}\cdots x_k^{d_k}$ and $U:=\operatorname{Vol}-c$.
Then the invariance of $\Sigma$ means that the inner product $(\nabla U, {\bf V})$ vanishes on it.
Indeed
$$
(\nabla U, {\bf V})=\sum_{i=1}^k \frac{\partial U}{\partial x_i}f_i=(U+c)\sum_{i=1}^k \frac{d_i}{x_i}f_i=
-2(U+c)\sum_{i=1}^k d_i\left({\bf r}_i-\frac{S_{\bf g}}{d}\right)=0
$$
due to $S_{\bf g}=\sum_{i=1}^kd_i{\bf r}_i$ and $d=\sum_{i=1}^kd_i$ (see also Lemma~2 in~\cite{Ab_RM}).

\smallskip

2) The invariance of $I_3$ means that  the vector field ${\bf V}$ is parallel to the tangent of $I_3$ on every its point. The direct calculations confirm such a property:
$$
\dfrac{d}{d\tau}\left(c\tau^{4-2n}\right)
=-2c(n-2)\tau^{-2n+3}=
\frac{f_3}{f_1}\bigg|_{I_3}=
\frac{f_3}{f_2}\bigg|_{I_3}.
$$

\smallskip

3) Can be proved in the same manner. Observe only that $I_1, I_2$ and $I_3$ can be obtained  by cyclic permutations in $(\tau,\tau,c\tau^{-2})$  in the case $n=3$. Moreover,
${\bf x}^0=\big(\sqrt[3]{c}, \sqrt[3]{c}, \sqrt[3]{c}\big)$  by Lemma~\ref{singpo_Stiefel}.
But for  $n\ge 4$ the curves $I_1$ and $I_2$ can not pass through the point ${\bf x}^0$ and hence can not be invariant.

\smallskip

4) Choose $I_1$. Substituting  $x_1= c^{\tfrac{1}{n-2}}\tau^{\tfrac{1-n}{n-2}}$, $x_2=\tau$ and $x_3=\tau$
into $\Pi_1$ and $\Pi_2$ we obtain
\begin{equation}\label{Pi_1|I_1}
\Pi_1\big|_{I_1}=\tau\left(2-c^{\tfrac{1}{n-2}} \tau^{-\tfrac{1}{n-2}}\right), \qquad
\Pi_2\big|_{I_1}=c^{\tfrac{1}{n-2}} \tau^{-\tfrac{n-1}{n-2}}.
\end{equation}
Analogously ${\bf r}_1$ and ${\bf r}_2$ take the following  forms on  $I_1$:
\begin{equation}\label{r_i|I_1}
\begin{array}{l}
{\bf r}_1\big|_{I_1}=\dfrac{1}{4(n-2)} \left(2(n-3)
+c^{\tfrac{2}{n-2}}\tau^{-\tfrac{2}{n-2}}\right) \tau^{\tfrac{n-1}{n-2}} c^{-\tfrac{1}{n-2}},\\
{\bf r}_2\big|_{I_1}=\dfrac{1}{4(n-2)\tau}\left(2(n-2)-c^{\tfrac{1}{n-2}} \tau^{-\tfrac{1}{n-2}}\right).
\end{array}
\end{equation}
Then $\Pi_2\big|_{I_1}>0$ and ${\bf r}_1\big|_{I_1}>0$ for $\tau>0$ and $n\ge 3$.
But $\Pi_1\big|_{I_1}>0$ only for $\tau>\tau_1$, where $\tau_1:=2^{2-n}c$ is the single root of
the equation $2-c^{\tfrac{1}{n-2}} \tau^{-\tfrac{1}{n-2}}=0$.

Note that ${\bf r}_2\big|_{I_1}>0$ for all $n\ge 3$ if $\tau>\tau_1$:
$$
2(n-2)-c^{\tfrac{1}{n-2}} \tau^{-\tfrac{1}{n-2}}=2(n-3)+2-c^{\tfrac{1}{n-2}}\tau^{-\tfrac{1}{n-2}}>0.
$$

By symmetry for $I_2$ we have
\begin{equation}\label{r_i|I_2}
\Pi_1\big|_{I_2}=\Pi_2\big|_{I_1}, \qquad \Pi_2\big|_{I_2}=\Pi_1\big|_{I_1}, \qquad
{\bf r}_1\big|_{I_2}={\bf r}_2\big|_{I_1}, \qquad {\bf r}_2\big|_{I_2}={\bf r}_1\big|_{I_1}.
\end{equation}

Consider now   $I_3$. Then
\begin{eqnarray*}
\Pi_1\big|_{I_3}=\Pi_2\big|_{I_3}=c\tau^{-2(n-2)}>0, \\
{\bf r}_1\big|_{I_3}={\bf r}_2\big|_{I_3}=\frac{1}{4(n-2)\tau}\left(2n-4-c\tau^{3-2n}\right).
\end{eqnarray*}
Therefore ${\bf r}_1\big|_{I_3}={\bf r}_2\big|_{I_3}>0$ for  $n\ge 3$ and
$\tau>\tau_2:=\left(\dfrac{2n-4}{c}\right)^{\tfrac{1}{3-2n}}$.
\end{proof}

\begin{remark}
As it follows from~\eqref{r_i|I_1}  $I_1$ intersects the cone ${\bf r}_2=0$
at the point $\left(c^{\tfrac{1}{n-2}}\,\tau_0^{\tfrac{1-n}{n-2}}, \tau_0, \tau_0\right)$,
where $\tau_0$ is defined from the equation $c^{\tfrac{1}{n-2}}\,\tau^{-\tfrac{1}{n-2}}=2(n-2)$
for all $n\ge 3$.
However this point  does not belong to the set  $\mathscr{R}_{+}$ due to
$\Pi_1\big|_{I_1}=\tau\left(2-c^{\tfrac{1}{n-2}}\,\tau^{-\tfrac{1}{n-2}}\right)\le 0$
 at $\tau=\tau_0$.
Indeed $2-2(n-2)=6-3n\le 0$ for $n\ge 3$.
Actually here we deal with the intersection of $I_1$ with another component~$\Gamma_2^{-}$ of the cone~${\bf r}_2=0$   establishing that $\Gamma_2^{-}$ is the extra component which
has no relation to the set  $\mathscr{R}_{+}$
(the lower surface in aquamarine color in Figure~\ref{pic-1}).
Thus $I_1$ does not intersect the useful component~$\Gamma_2$
of~${\bf r}_2=0$ (see~\eqref{eqn_181124}).
The similar conclusion can be obtained for $I_2$ interchanging the indices by symmetry.
\end{remark}

\begin{lemma}\label{singpo_Stiefel}
The following assertions hold for the system~\eqref{three_equat_Stif}  at a fixed $c>0$:
\begin{enumerate}
\item
For every $n\ge 3$ the system~\eqref{three_equat_Stif} has the unique family of one parametric singular (equilibrium) points given by the formula
$$
{\bf x}^0=(q, q, \kappa q)\in \mathscr{R}_{+}, \mbox{~~ where ~~}  \kappa:=2(n-2)(n-1)^{-1}, \quad q\in \mathbb{R}_{+}.
$$
The actual value $q=q_0$ of the parameter $q$ at which ${\bf x}^0$  belongs to $\Sigma$  can be found as
$$
q_0=\sqrt[2n-3]{c\kappa^{-1}}.
$$
\item
If $n=3$ then $\Sigma$ is the stable manifold of~${\bf x}^0$
with the tangent space  $E^s=\operatorname{Span}\{(-1,1,0), (-1,0,1)\}$ and, hence,  ${\bf x}^0\in \Sigma$ is a stable node.
Moreover, $I_1, I_2$ and $I_3$ are one dimensional stable submanifolds for~${\bf x}^0$.
\item
If $n\ge 4$ then ${\bf x}^0$ is a saddle
with the tangent spaces $E^s=\operatorname{Span}\left\{\big(1, 1, -4(n-2)^2(n-1)^{-1}\big)\right\}$ and $E^u=\operatorname{Span}\{(-1,1,0)\}$
of corresponding manifolds (separatrices)  on~$\Sigma$.
Moreover, the stable separatrice is exactly $I_3$ for all $n\ge 4$.
\item
For all $n\ge 3$ the singular point ${\bf x}^0$ admits also the center manifold
with the tangent $E^c=\operatorname{Span}\big\{(1, 1, \kappa)\big\}$.
\end{enumerate}
\end{lemma}

\begin{proof} Fix $c>0$ and use ideas of~\cite{Ab_RM}.
1) It should be noted that singular (equilibrium) points of the system~\eqref{decom_ricciflow_general}
is also  invariant Einstein metrics of a considered space $G/H$  and vice versa.
Such a conclusion easily follows from the observation that the equalities $f_1=\cdots =f_k=0$ and
${\bf r}_1=\cdots ={\bf r}_k=d^{-1}S_{\bf g}$ are equivalent.
In particular, according to~\cite{Arvan2, Kerr} the Stiefel manifold $\operatorname{SO}(n)/\operatorname{SO}(n-2)$
admits the unique $\operatorname{SO}(n)$-invariant Einstein metric
$\left(1, 1, \kappa\right)$ up to scale
for every $n\ge 3$ which is also the unique  singular point of the system~\eqref{three_equat_Stif}.
The actual value $q_0$ providing ${\bf x}^0\in \Sigma$
can easily be found from the condition
$$
\operatorname{Vol}(q, q, \kappa q)=c.
$$
To establish that ${\bf x}^0\in \mathscr{R}_{+}$ for all $n\ge 3$ it suffices
to check  the definition of the set $\mathscr{R}_{+}$:
$$
{\bf r}_1({\bf x}^0)={\bf r}_2({\bf x}^0)=\frac{2(n-2)-\kappa}{4q(n-2)}=
\frac{n-2}{2q(n-1)}>0, \qquad {\bf r}_3({\bf x}^0)=\frac{\kappa}{4q}>0.
$$

Denote by  $\chi(\lambda, {\bf x})$ the characteristic polynomial  of the
Jacobian matrix~ ${\bf J}({\bf x})$ of the vector field~${\bf V}({\bf x})$
associated with the differential system~\eqref{three_equat_Stif},
${\bf x}=(x_1, x_2, x_3)$.
Then the cubic equation
$$
\chi(\lambda, {\bf x}^0)= \lambda^3+\frac{n-1}{q(n-2)^2}\,\lambda^2-\frac{n^2-5n+5}{q^2(n-2)^2}\,\lambda=0
$$
admits three real roots
$$
\lambda_1=(n^2-5n+5)(n-2)^{-2}q^{-1}, \qquad \lambda_2=-q^{-1}, \qquad \lambda_3=0
$$
for ${\bf x}={\bf x}^0=\left(q, q, \kappa q\right)$.

2) If $n=3$ then $\lambda_1=\lambda_2=-q^{-1}<0$. Therefore
${\bf x}^0$ is a stable (attracting) node on every~$\Sigma$
with the stable manifold $W^s=\Sigma$ admitting the tangent space
spanned by the eigenvectors  $(-1,1,0)$ and $(-1,0,1)$ corresponding to the
eigenvalue $\lambda_1$ of multiplicity $2$.

In addition,  $I_1, I_2$ and $I_3$ are  invariants at $n=3$  by Lemma~\ref{invar_Stiefel}. Moreover $I_1, I_2$ and $I_3$ are also the subsets of the stable manifold~$\Sigma$ and hence each of them
is stable.

3) For  $n\ge 4$  we have $\lambda_1\lambda_2<0$. Therefore
${\bf x}^0$ is a saddle on~$\Sigma$  with one dimensional
stable and unstable manifolds (separatrices) both contained in $\Sigma$.
The  tangent space $E^s$ of the stable separatrice is spanned by
the eigenvector $\big(1, 1, -4(n-2)^2(n-1)^{-1}\big)$ corresponding to $\lambda_2=-q^{-1}<0$.
The eigenvector $(-1,1,0)$ corresponds to  $\lambda_1>0$.

In addition,  $I_3$ maintains the property
to be invariant for all $n\ge 4$ by Lemma~\ref{invar_Stiefel} and it is the stable manifold of the saddle ${\bf x}^0$ in fact.
Indeed since $\dot{x}_1(\tau)\equiv 1$, $\dot{x}_2(\tau)\equiv 1$, $\dot{x}_3(\tau)=(4-2n)c\tau^{3-2n}$
we have
\begin{eqnarray*}
\dot{x}_1(q_0)&=& \dot{x}_2(q_0)=1, \\
\dot{x}_3(q_0)&=&(4-2n) c q_0^{3-2n}=(4-2n) \kappa=-4(n-2)^2(n-1)^{-1}
\end{eqnarray*}
for $\tau=q_0$ corresponding to the saddle point ${\bf x}^0=(q_0,q_0,\kappa q_0)\in I_3\subset \Sigma$. Therefore the vector $\big(1, 1, -4(n-2)^2(n-1)^{-1}\big)$
is tangent to $I_3$ at  ${\bf x}^0$.
On the other hand this vector spans the stable eigenspace~$E^s$
of~${\bf x}^0$ as shown above.
Therefore the stable manifold (separatrice) $W^s$ of ${\bf x}^0$ coincides
with~$I_3$ for every $n\ge 4$.

4) In final,  $\lambda_3=0$ is responsible for the center (slow) manifold
with the tangent space spanned by the eigenvector $(1, 1, \kappa)$ for all  $n\ge 3$.
\end{proof}

\begin{remark}
Actually we are interested in the special case $c=1$
which corresponds to the metrics~\eqref{metricStiffel_V_2(R^n)} of the unit volume and,
therefore, has a purely geometric meaning.
Studies of~\eqref{three_equat_Stif} on $\Sigma$ with an arbitrary~$c>0$ can easily be reduced to the ``geometric'' case $c=1$ by the change of variables $x_i=X_i \sqrt[3]{c}$ and $t=\tau \sqrt[3]{c}$ in~\eqref{three_equat_Stif} basing on homogeneity of the functions $f_i$ and
autonomy of the system~\eqref{three_equat_Stif} (see also~\cite{Ab_RM}).
So  in the sequel  we assume that $c=1$ without loss of generality.
\end{remark}

\subsection{Structural properties of the set $\Sigma \cap\mathscr{R}_{+}$}

We need also new sets $\gamma_1=\Sigma \cap \widetilde{\Gamma}_1$,
$\gamma_2=\Sigma \cap \widetilde{\Gamma}_2$ (the curves in red and teal color respectively
in Figure~\ref{pic-1}) and $\pi_1=\Sigma \cap \Pi_1$, $\pi_2=\Sigma \cap \Pi_2$
(the curves in blue color there).

\begin{lemma}\label{lem_Stif_curves}
For all $n\ge 3$ the set $\Sigma\cap \mathscr{R}_{+}$  is bounded by the smooth and connected curves
$\pi_1, \pi_2, \gamma_1$ and~$\gamma_2$
such that:

1)  $\pi_1 \cap \pi_2= \pi_1 \cap \gamma_1= \pi_2 \cap \gamma_2=\emptyset$ and
the components  in the pairs $(\pi_1, \pi_2)$, $(\pi_1,  \gamma_1)$ and $(\pi_2,  \gamma_2)$
approach each other at infinity as close as we like (shortly  $\pi_1 \rightarrow \pi_2$,
$\pi_1 \rightarrow \gamma_1$ and   $\pi_2 \rightarrow \gamma_2$);

2) The curves  $\gamma_1$ and $\gamma_2$ admit the single common point $P_{12}\big(\overline{p}, \overline{p}, (2n-4)\overline{p}\big)$, where
$\overline{p}=\left(\dfrac{c}{2n-4}\right)^{\tfrac{1}{2n-3}}$.
\end{lemma}

\begin{proof}
{\it Intersections and long time behaviors of the curves $\pi_i, \gamma_i$, $i=1,2$}.
Since $\gamma_1\cap \gamma_2=(\widetilde{\Gamma}_1\cap \widetilde{\Gamma}_2)\cap \Sigma$ and $\widetilde{\Gamma}_1\cap \widetilde{\Gamma}_2=L$ by Lemma~\ref{lem_Stif_surfes} the set $\gamma_1\cap \gamma_2$
consists of the single point $P_{12}$, where $\overline{p}$ is defined as the unique root of the equation $p^{n-2}p^{n-2}(2n-4)p=c$ for a given $n\ge 3$.

 By analogy  $\pi_1 \cap \pi_2= \pi_1 \cap \gamma_1= \pi_2 \cap \gamma_2=\emptyset$
  easily follows from
 $\Pi_1 \cap \Pi_2= \Pi_1 \cap \widetilde{\Gamma}_1= \Pi_2 \cap \widetilde{\Gamma}_2=\emptyset$  also known from Lemma~\ref{lem_Stif_surfes} for all $n\ge 3$.
To prove that $\pi_1$ and~$\gamma_1$ approximate each other at infinity  we use the approach developed in~\cite{Ab_RM, Ab24}:
we will  show $\pi_1\rightarrow I_2$ and $\gamma_1\rightarrow I_2$
instead of $\pi_1 \rightarrow \gamma_1$
(note that it is quite difficult
to derive it directly from the corresponding systems
which define $\pi_1$ and~$\gamma_1$).
According to Lemma~\ref{invar_Stiefel} we know that
$I_1, I_2,I_3\subset \mathscr{R}_{+}$ for sufficiently large $\tau$ (more precisely, for all $\tau>\max\{\tau_1, \tau_2\}$).

Moreover, Lemma~\ref{invar_Stiefel} also implies  the following limits for every fixed $n\ge 3$ (see formulas~\eqref{Pi_1|I_1}, \eqref{r_i|I_1} and \eqref{r_i|I_2}):
\begin{eqnarray*}
\lim_{\tau \to +\infty}{\bf r}_1\big|_{I_2} &=&
\lim_{\tau \to +\infty}\frac{1}{4(n-2)\tau}\left(2(n-2)-\tau^{-\tfrac{1}{n-2}}\right)=0,\\
\lim_{\tau \to +\infty}\Pi_1\big|_{I_2}&=&  \lim_{\tau \to +\infty}\tau^{-\tfrac{n-1}{n-2}}=0.
\end{eqnarray*}
Therefore $\pi_1 \rightarrow \gamma_1$ at infinity as close as we want.
Then $\pi_2 \rightarrow \gamma_2$ is clear from symmetry.
In final $\Pi_1\big|_{I_3}=\Pi_2\big|_{I_3}=\tau^{-2(n-2)}\rightarrow 0$ as $\tau\to +\infty$ implying
$\pi_1 \rightarrow \pi_2$.

\smallskip

{\it Smoothness and connectedness of the curves $\pi_i, \gamma_i$, $i=1,2$}.
Substituting $x_2=tx_1$, where $t>0$, into ${\bf r}_3=0$ we easily get  $(t-1)^2x_1^2-x_3^2=0$.
Taking into account the condition $x_1^{n-2}x_2^{n-2}x_3=c$ (the equation of the surface $\Sigma$)
we obtain the  parametric equations
$$
x_1=\phi_1(t)=c^{\frac{1}{2n-3}}\, t^{\frac{2-n}{2n-3}}\, |t-1|^{\frac{1}{3-2n}}, \qquad
x_2=\phi_2(t)=t\,\phi_1(t), \qquad
x_3=\phi_3(t)=|t-1|\,\phi_1(t)
$$
for~$\pi_1$ and $\pi_2$ with $t\in (0,+\infty)\setminus \{1\}$ so that the interval  $(0,1)$  corresponds to $\pi_1$ and
$(1,+\infty)$  corresponds to $\pi_2$.
By analogy putting $x_2=tx_3$ in ${\bf r}_1=0$ the equations
$$
x_3=\psi_3(t)=c^{\frac{1}{2n-3}}\, t^{-\frac{n-2}{2n-3}}\, \Psi(t)^{-\frac{n-2}{2n-3}}, \quad
x_2=\psi_2(t)=t\,\psi_3(t), \quad
x_1=\psi_1(t)=\psi_3(t)\Psi(t), \quad  t\in (0,l),
$$
can be obtained for the curve $\Sigma\cap \Gamma_1$ with  $\Psi(t)=\sqrt{t^2-2(n-2)t+1}$  and  $l$ is given in~\eqref{eqn_211124}. The  representation $x_1=\psi_2(t)$, $x_2=\psi_1(t)$, $x_3=\psi_3(t)$
is clear for the curve $\Sigma\cap \Gamma_2$ due to symmetry, $t\in (0,l)$.
It is clear  that   $\gamma_1\subset \Sigma\cap \Gamma_1$ and
$\gamma_2 \subset \Sigma\cap \Gamma_2$ and hence $\gamma_1$ and $\gamma_2$
can be parameterized by the same functions $\psi_1, \psi_2, \psi_3$, but with
$t\in \Delta$, where $\Delta$ is some interval such that $\Delta\subset (0,l)$.
To clarify $\Delta$ recall  Lemma~\ref{lem_Stif_surfes} according to which
the curves $\Sigma\cap \Gamma_1$ and  $\Sigma\cap \Gamma_2$   leave ``tails'' (extra pieces)
not relating to the set  $\Sigma \cap \mathscr{R}_{+}$ after their intersection
in the  case  $n\ge 4$ and  coincide if $n=3$.
So to eliminate those extra pieces take into account $x_1=x_2$ at the unique point $P_{12}\in \gamma_1 \cap \gamma_2$ which implies the equation $\Psi(t)=t$ with the unique root $\widetilde{t}=(2n-4)^{-1}$.
In what follows  that  $\Delta=\big(0,\widetilde{t}\,\big]$.

Now it is easy to see that the functions $\phi_1, \phi_2$ and $\phi_3$ are differentiable
on the sets $(0,1)$ and $(1,+\infty)$.
Therefore the curves $\pi_1$  and $\pi_2$ are smooth.
The curve   $\pi_1$ must be connected as the image of the connected set $(0,1)$ under the continuous function  $t \mapsto \big(\phi_1(t),\phi_2(t),\phi_3(t)\big)$, $t\in (0,1)$.
So is $\pi_2$ being the continuous image of the connected set $(1,+\infty)$.
Smoothness and connectedness of  $\gamma_1$ and $\gamma_2$
follow analogously from  the differentiability of the functions $t \mapsto \big(\psi_1(t),\psi_2(t),\psi_3(t)\big)$ and
$t \mapsto \big(\psi_2(t),\psi_1(t),\psi_3(t)\big)$,
where $t\in \Delta$ and $\Delta$ is a connected set.
\end{proof}

\begin{figure}[h!]
\centering
\includegraphics[width=0.45\textwidth]{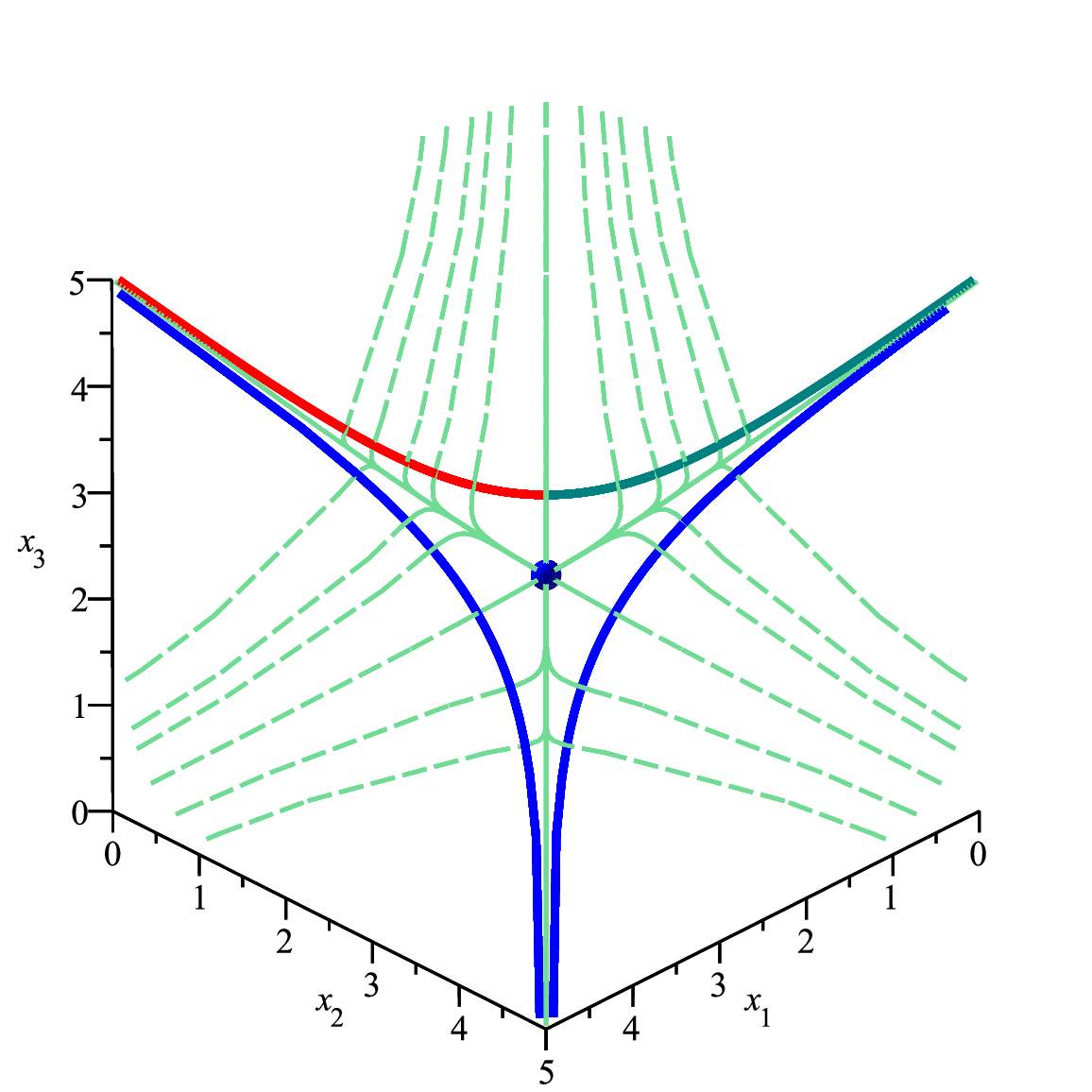}
\includegraphics[width=0.45\textwidth]{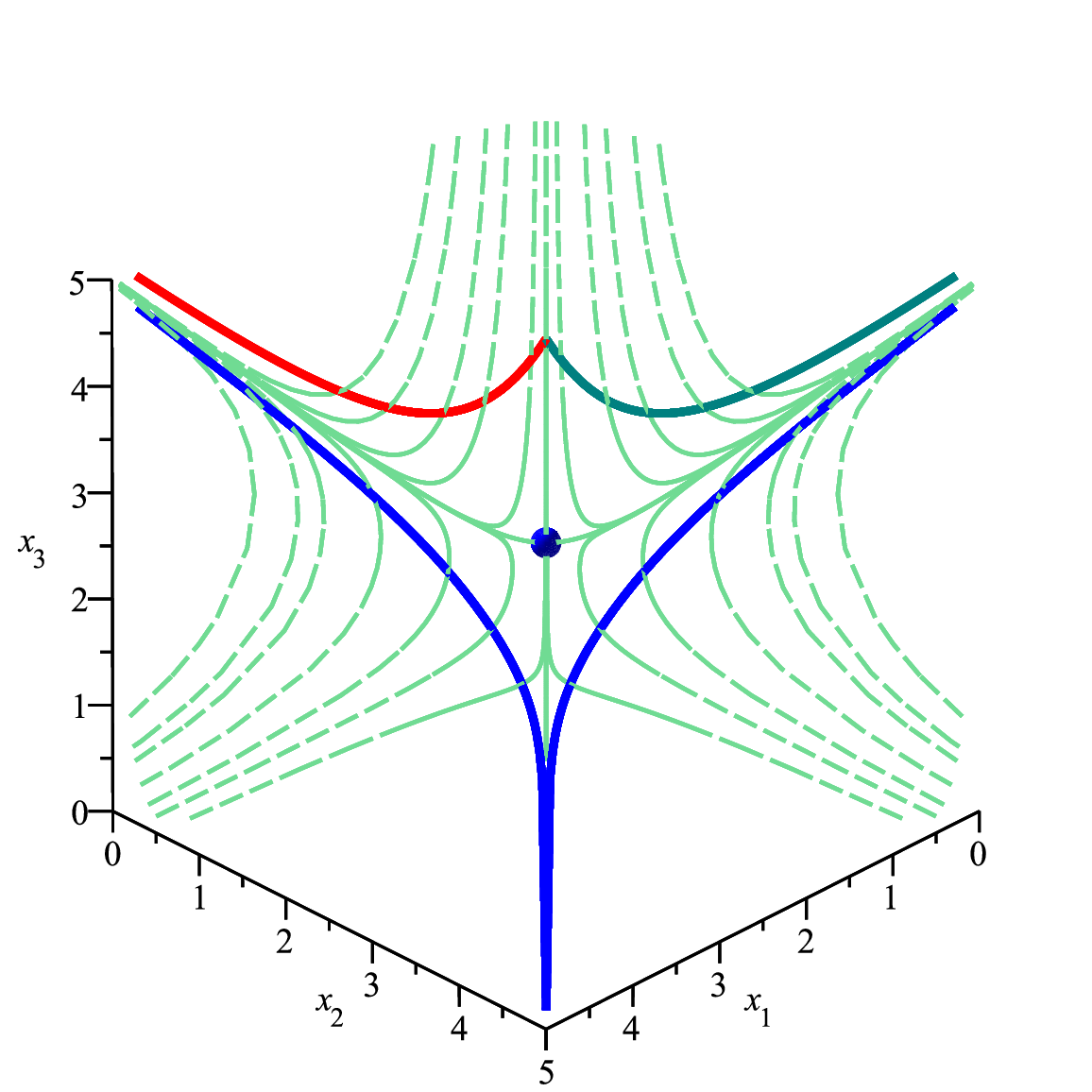}
\caption{The dynamics  of the system~\eqref{three_equat_Stif} on the invariant set $\Sigma$
for  $n=3$ (the left panel) and $n=4$ (the right panel)}
\label{pic1}
\end{figure}

\bigskip

\begin{proof}[Proof of Theorem~\ref{thm_371124}]
{\it Step   1}. Any trajectory of~\eqref{three_equat_Stif} originating in~$\mathscr{R}_{+}$ {\it remains} there forever.
Indeed by Lemma~\ref{lem_241024_2}   the vector field ${\bf V}$
associated with the  system~\eqref{three_equat_Stif}
is directed  into~$\mathscr{R}_{+}$ on every point of~$\partial(\mathscr{R}_{+})$ for all $n\ge 3$.
Therefore  no trajectory of~\eqref{three_equat_Stif}
can leave $\mathscr{R}_{+}$: if ${\bf x}(0)\in \mathscr{R}_{+}$ then
${\bf x}(t)\in \mathscr{R}_{+}$ for all $t>0$.
According to the definition of the set $\mathscr{R}_{+}$ this means that~\eqref{ricciflow}
preserves the positivity of the Ricci curvature of metrics~\eqref{metricStiffel_V_2(R^n)}
on $\operatorname{SO}(n)/\operatorname{SO}(n-2)$:
metrics with positive Ricci curvature can be evolved only into metrics with positive Ricci curvature.

{\it Step 2}. Any trajectory originating in the exterior of $\mathscr{R}_{+}$ {\it enters} $\mathscr{R}_{+}$ in finite time. To prove this assertion we use Lemma~\ref{singpo_Stiefel}.
It suffices to study~\eqref{three_equat_Stif} on $\Sigma \cap \mathscr{R}_{+}$,
 where $\Sigma$  are invariant surfaces of~\eqref{three_equat_Stif}
responsible for dominant motions of its trajectories corresponding to
nonzero eigenvalues $\lambda_1$ and $\lambda_2$
according to Lemmas~\ref{invar_Stiefel} and~\ref{singpo_Stiefel}.
Any movement caused by $\lambda_3=0$  for  all $n\ge 3$ can be neglected because it can only occur within the domain $\mathscr{R}_{+}$ itself along the eigenvector $(1,1,\kappa)$ (as slow transitions between different invariant surfaces):
$$
{\bf r}_1 (1,1,\kappa)={\bf r}_2 (1,1,\kappa)=\frac{n-2}{2(n-1)}>0, \qquad
{\bf r}_3 (1,1,\kappa)=\frac{\kappa}{4}>0.
$$

By  Lemma~\ref{singpo_Stiefel} for every $n\ge 3$  there exists $q=q_0$ such that
${\bf x}^0\in \Sigma\cap \mathscr{R}_{+}$. Therefore
every trajectory of~\eqref{three_equat_Stif} originated in~$\Sigma\cap \operatorname{ext}\mathscr{R}_{+}$ must intersect the boundary $\Sigma\cap \partial(\mathscr{R}_{+})$ in finite time and enter the set $\Sigma\cap \mathscr{R}_{+}$
governing by the stable manifold $W^s=\Sigma$ (and by the stable submanifolds $I_1, I_2$ and $I_3$ as well) of the unique stable node~${\bf x}^0$
in the case $n=3$  and by the separatrices $W^s=I_3\subset \Sigma$ and $W^u\subset \Sigma$ of~${\bf x}^0$ being the unique saddle for every $n\ge 4$ (see Figure~\ref{pic1}).
The domain~$\Sigma \cap \mathscr{R}_{+}$ is able to receive all trajectories
due to the fact that it is unbounded with the boundary~$\Sigma \cap \partial(\mathscr{R}_{+})$ consisting of
the curves $\pi_1, \pi_2, \gamma_1$ and $\gamma_2$ in $\mathbb{R}^3$ such that $\pi_1\cap \pi_2=\pi_1\cap \gamma_1=\pi_2\cap \gamma_2=\emptyset$ according to Lemma~\ref{lem_Stif_curves}   (which is also quite consistent with the facts established above,
indeed, if $\Sigma \cap \partial(\mathscr{R}_{+})$ was a bounded set, then
in order to satisfy Lemma~\ref{lem_241024_2}, the dynamics of the system~\eqref{three_equat_Stif} should be different from those described in Lemma~\ref{singpo_Stiefel} contradicting  it).
Another significant circumstance
what causes all trajectories to enter~$\mathscr{R}_{+}$ and
remain there forever for all $n\ge 4$ is that the separatrices $W^s$ and  $W^u$ are subsets in~$\mathscr{R}_{+}$ at infinity since no of them can reach (intersect or touch) the boundary~$\Sigma\cap \partial(\mathscr{R}_{+})$ according to Lemma~\ref{lem_241024_2}.

Thus we proved that on the Stiefel manifold~$\operatorname{SO}(n)/\operatorname{SO}(n-2)$, $n\ge 3$,
the normalized Ricci flow~\eqref{ricciflow} evolves   all metrics~\eqref{metricStiffel_V_2(R^n)}  (of mixed or positive Ricci curvature) into metrics with positive Ricci curvature.
Theorem~\ref{thm_371124} is proved.
\end{proof}

\begin{figure}[h!]
\centering
\includegraphics[width=0.45\textwidth]{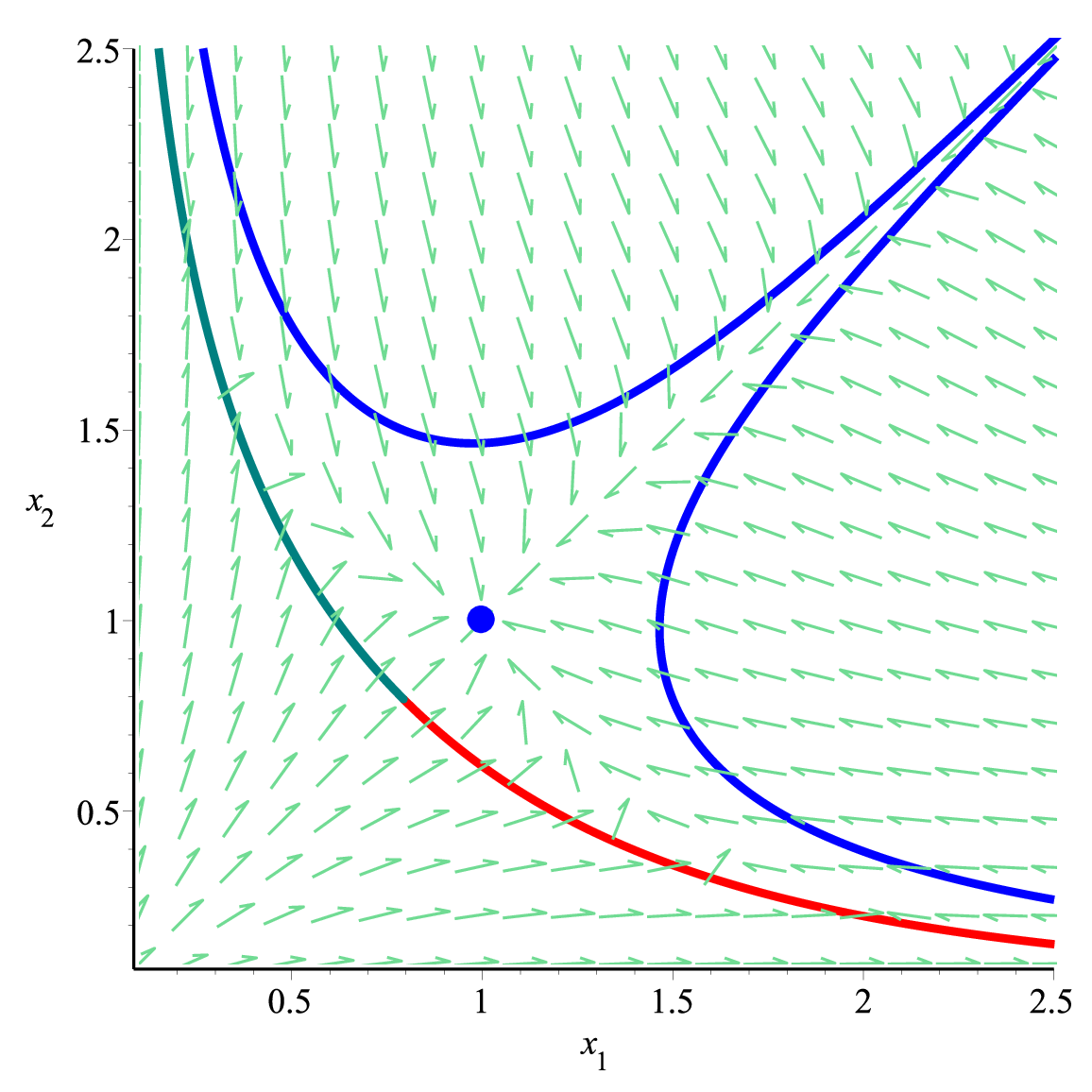}
\includegraphics[width=0.45\textwidth]{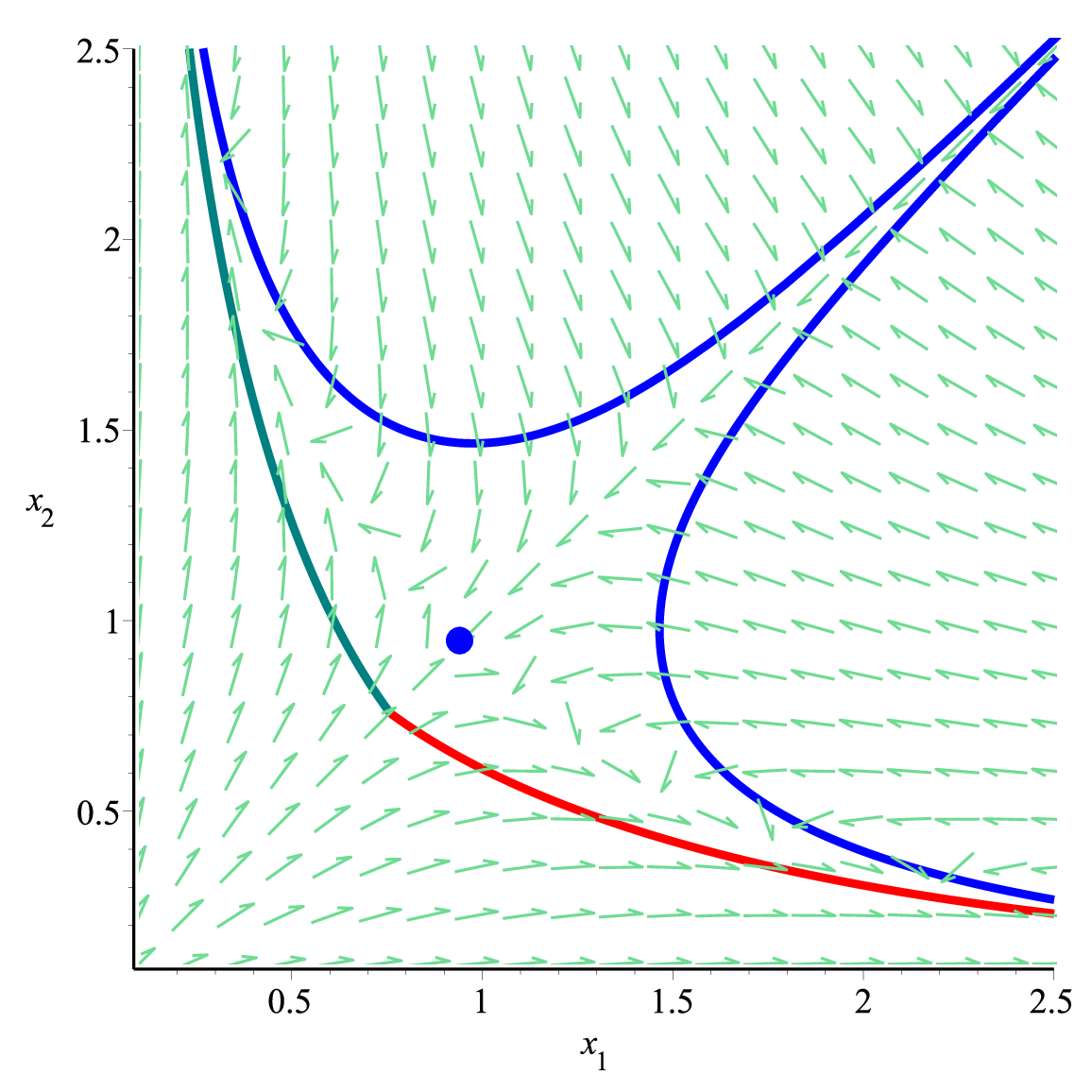}
\caption{The phase portrait of the planar system~\eqref{two_equat_Stif}
for  $n=3$ (the left panel) and $n=4$ (the right panel)}
\label{pic3}
\end{figure}

\section{Planar illustrations}

Clearly  the volume function $\operatorname{Vol}=x_1^{n-2}x_2^{n-2}x_3$
is a first integral of~\eqref{three_equat_Stif}.
Basing on this fact and using the techniques developed in~\cite{AANS1}
the system~\eqref{three_equat_Stif} could  equivalently be reduced to the following planar system:
\begin{equation}\label{two_equat_Stif}
\dfrac {dx_1}{dt} = \widetilde{f}_1(x_1,x_2), \qquad  \dfrac {dx_2}{dt} = \widetilde{f}_2(x_1,x_2),
\end{equation}
where
$\widetilde{f}_i(x_1,x_2):= f_i\left(x_1,x_2,\varphi(x_1,x_2)\right)$
and $\varphi(x_1,x_2):=(x_1x_2)^{2-n}$.
The corresponding  singular point of~
\eqref{two_equat_Stif} is
$\widetilde{{\bf x}}^0=(q_0, q_0)$, where $q_0=\sqrt[2n-3]{(n-1)(2n-4)^{-1}}$.

On the other hand it is well known that
$\chi(\lambda, {\bf x}^0)=\lambda^3-\rho\lambda^2+\delta\lambda$,
where $\rho=\operatorname{trace}\widetilde{{\bf J}}(\widetilde{{\bf x}}^0)$ and
$\delta=\operatorname{det}\widetilde{{\bf J}}(\widetilde{{\bf x}}^0)$
are respectively the trace and the determinant of
the Jacobian matrix $\widetilde{{\bf J}}(\widetilde{{\bf x}})$ of the system~\eqref{two_equat_Stif} evaluated at its singular point $\widetilde{{\bf x}}=\widetilde{{\bf x}}^0$
(see also \cite{AANS1}).
Then consequently
$$
\delta=-\frac{n^2-5n+5}{(n-2)^2q^2}, \qquad
\rho=-\frac{n-1}{(n-2)^2q}<0, \qquad \sigma=\rho^2-4\delta=\frac{(2n-3)^2(n-3)^2}{(n-2)^4q^2}\ge 0.
$$
According to the theory of planar dynamical systems $\widetilde{{\bf x}}^0$ is a hyperbolic stable node of~\eqref{two_equat_Stif} for $n=3$
since $\delta>0$, $\rho<0$ and $\sigma=0$.
Clearly, $\widetilde{{\bf x}}^0$ is a hyperbolic saddle of~\eqref{two_equat_Stif} for all $n\ge 4$
since $\delta<0$. These results are illustrated in Figure~\ref{pic3}.

\begin{figure}[h!]
\centering
\includegraphics[width=0.45\textwidth]{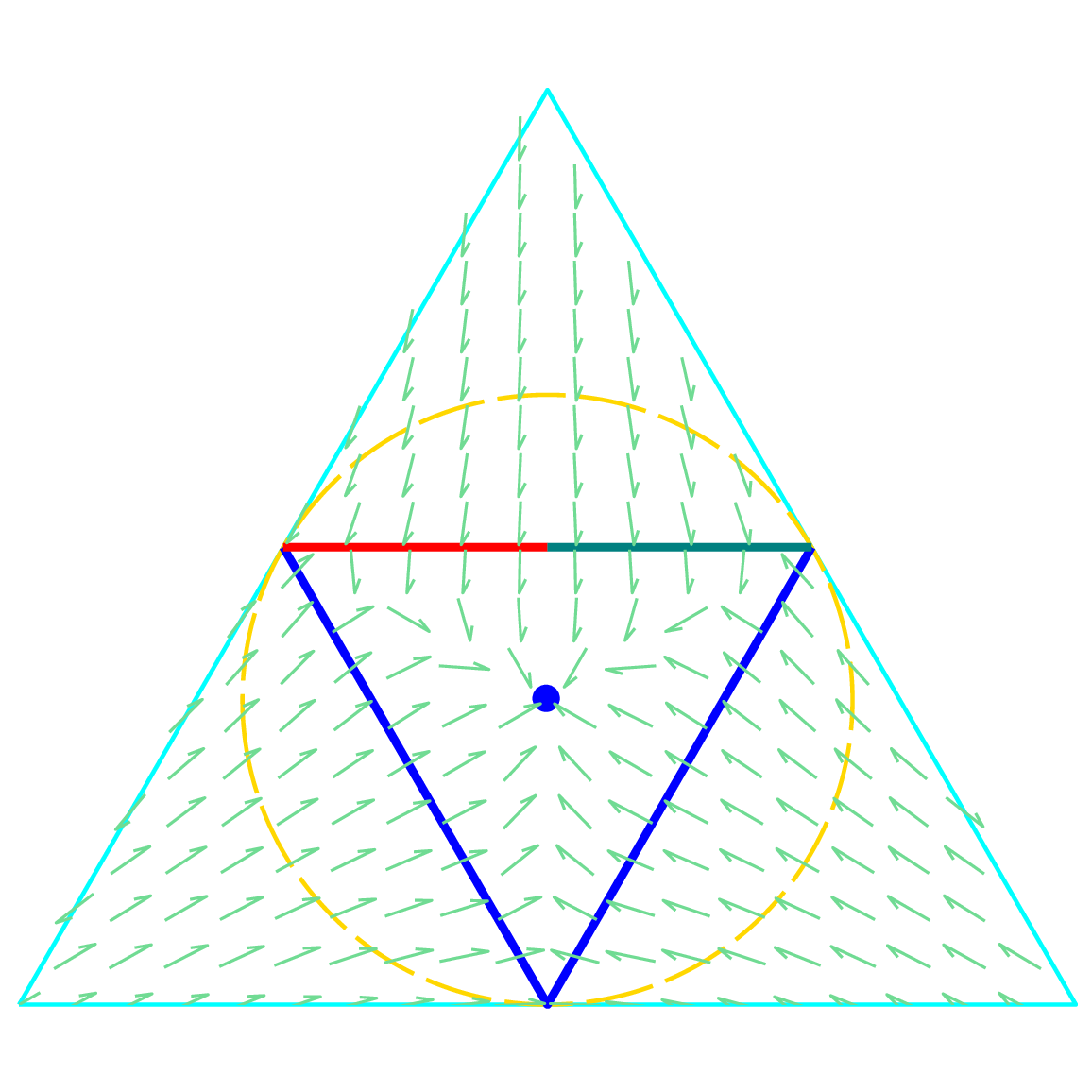}
\includegraphics[width=0.45\textwidth]{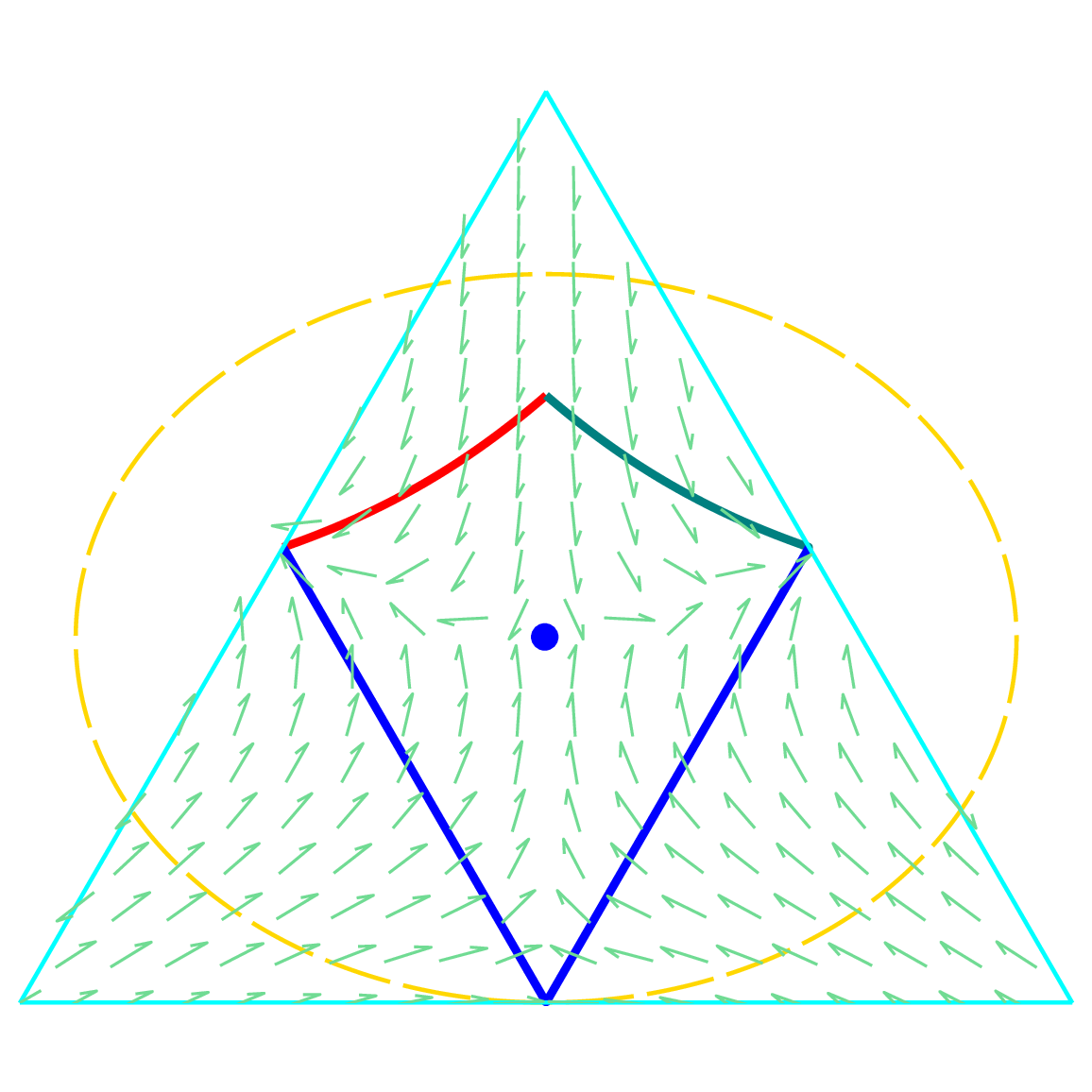}
\caption{The phase portrait of~\eqref{three_equat_Stif} on the plane $x_1+x_2+x_3=1$
at $n=3$  (the left panel) $n=4$ (the right panel)
}
\label{pic5}
\end{figure}

\

\medskip

Using the idea suggested by Professor Wolfgang Ziller in~\cite{AN}
the dynamics of~\eqref{three_equat_Stif} can be illustrated on the plane $x_1+x_2+x_3=1$ (denote it $\Pi_0$)
preserving the dihedral symmetry of the problem.
In Figure~\ref{pic5} the relevant results are depicted,
where the edges of the big triangle
correspond to $x_i=0$ on $\Pi_0$;
the  curves   $\Pi_0\cap \widetilde{\Gamma}_1$,
$\Pi_0\cap \widetilde{\Gamma}_2$ and $\Pi_0\cap \Pi_1$, $\Pi_0 \cap \Pi_2$
are depicted in red, teal and blue colors respectively;
the yellow circle (or ellipse) represents  the curve defined by the systems of the equations
$x_0^2-(2n-4)(x_1+x_2)x_0+(x_1-x_2)^2=0$ and $x_1+x_2+x_3=1$,
where the first one is equivalent to $S_{{\bf g}}=0$.
These illustrations confirm the well known general fact that the normalized Ricci flow
preserves the positivity of the scalar curvature
of invariant  metrics on every compact homogeneous space
(see~\cite{Ham2, Lauret}).

\section{Final remarks}

 The geometric problem considered in the article can represent some interest from the point of view of the dynamical systems. It is not easy to find the smallest invariant set (or  minimal invariant sets) of the flow.
Actually  in Theorem~\ref{thm_371124} we proved that
for every invariant set~$\Sigma$ of the system~\eqref{three_equat_Stif}
defined as $x_1^{n-2}x_2^{n-2}x_3=c$ there exists its invariant subset $\Sigma\cap \mathscr{R}_{+}$.
Similar results were obtained in \cite{Ab7, Ab_RM, Ab24, AN} for~\eqref{ricciflow}
reduced to a dynamical system on  generalized Wallach spaces,
see, for instance,
the case $a\in (1/6, 1/4)\cup (1/4,1/2)$ in~\cite[Theorem~3]{AN}, the case $a=1/6$ in~\cite[Theorem~4 ]{AN},
the case $a=1/4$ in~\cite[Theorem~3]{Ab7}
and the case $a\in [1/4,1/2)$ in~\cite[Theorem~2]{Ab_RM}
which concern invariance of certain sets related to generalized Wallach spaces
with $a_1=a_2=a_3=a$. The general case $a_i\ne a_j\ne a_k\ne a_i$ was studied
in the recent manuscript~\cite{Ab24}, where some additional conditions were found
on the parameters $a_1,a_2,a_3$  which provide the invariance of a set analogous to~$\mathscr{R}_{+}$.

\section*{Acknowledgments}

The author is indebted to Professor Yu.\,G.~Nikonorov for helpful discussions concerning the topic.

\vspace{10mm}

\bibliographystyle{amsunsrt}

\vspace{5mm}

\end{document}